\documentclass[12pt,a4paper]{amsart}

\usepackage{color}

\usepackage{amsmath,amsfonts,amssymb, amsthm}
\usepackage{mdwlist}

\usepackage{enumitem}
\usepackage{amscd}
\usepackage[all,cmtip]{xy}
\usepackage{hyperref}
\usepackage{fancyhdr}
\usepackage{url}
\usepackage{graphicx}

\newtheorem{thm}{Theorem}[section]
\newtheorem*{thm*}{Theorem}
\newtheorem{lm}[thm]{Lemma}
\newtheorem{prop}[thm]{Proposition}
\newtheorem{cor}[thm]{Corollary}
\theoremstyle{definition}
\theoremstyle{definition}\newtheorem{ex}[thm]{Example}
\theoremstyle{remark}\newtheorem{rem}[thm]{Remark}

\DeclareMathOperator{\spec}{Spec}

\DeclareMathOperator{\proj}{Proj}

\DeclareMathOperator{\Proj}{Proj}

\DeclareMathOperator{\Hilb}{Hilb}

\DeclareMathOperator{\Ann}{Ann}

\newcommand{\ra}{\rightarrow}

\newcommand{\func}[3]{\ensuremath{#1\mathpunct: #2\ra #3}}

\newcommand{\injfunc}[3]{\ensuremath{#1\mathpunct: #2\hookrightarrow #3}}
\newcommand{\iso}[3]{\ensuremath{#1\mathpunct:
    #2\stackrel{_\sim}{\longrightarrow} #3}}

\newcommand{\CM}{\ensuremath{\mathit{CM}} }
\newcommand{\opn}[1]{\operatorname{#1}}

\newcommand{\calI}{\ensuremath{{\mathcal{I}}}}

\newcommand{\calK}{\ensuremath{{\mathcal{K}}}}

\newcommand{\calN}{\ensuremath{{\mathcal{N}}}}
\newcommand{\calO}{\ensuremath{{\mathcal{O}}}}

\newcommand{\calX}{\ensuremath{{\mathcal{X}}}}

\newcommand{\p}{\ensuremath{\mathfrak{p}}}

\newcommand{\m}{\ensuremath{\mathfrak{m}}}

\newcommand{\bbN}{\ensuremath{{\mathbb{N}}}}
\newcommand{\bbP}{\ensuremath{{\mathbb{P}}}}

\newcommand{\bbZ}{\ensuremath{{\mathbb{Z}}}}

\newcommand{\e}{\varepsilon}

\begin{document}
\title{The Cohen--Macaulay space of 
  twisted cubics} \author{Katharina Heinrich} \address{KTH Royal
  Institute of Technology, Institutionen f\"or matematik, 10044
  Stockholm, Sweden} \email{kchal@math.kth.se}

\begin{abstract}
  In this work, we describe the Cohen--Macaulay space $\CM$ of twisted
  cubics parameterizing curves $C$ together with a finite map $\func i
  C \bbP^3$ that is generically a closed immersion and such that $C$
  has Hilbert polynomial $p(t)=3t+1$ with respect to $i$. We show that
  $\CM$ is irreducible, smooth and birational to one component of the
  Hilbert scheme of twisted cubics.
\end{abstract}

\maketitle
\thispagestyle{empty}

\section{Introduction}
 
A {\em twisted cubic} is a smooth, rational curve in $\bbP^3$ of
degree $3$ and genus $0$. It is projectively equivalent to the image
of the Veronese map $\bbP^1\ra \bbP^3$ mapping a point $[u:v]$ on the
line to the point in $\bbP^3$ with coordinates
$[u^3:u^2v:uv^2:v^3]$. Being the simplest example of a space curve,
these curves have been the object of interest in many problems in
algebraic geometry. Here we compare two modular compactifications of
the space $\calX$ of twisted cubics.

The first, and classical, modular compactification is given by the
{\em Hilbert scheme} $\Hilb_{\bbP^3}^{3t+1}$ parameterizing all closed
subschemes in $\bbP^3$ having Hilbert polynomial $p(t)=3t+1$. Piene
and Schlessinger gave in \cite{PS:twistedcubic} a detailed description of
$\Hilb_{\bbP^3}^{3t+1}$. It has two smooth irreducible components
$H_0$ and $H_1$ with generic points corresponding to a twisted cubic
and a smooth plane curve with an additional isolated point,
respectively. The component $H_0$ actually contains all curves in
$\Hilb_{\bbP^3}^{3t+1}$ that do not have an embedded or isolated
point, and especially all twisted cubics. Being a significantly
smaller compactification of $\calX$ than the whole Hilbert scheme
$\Hilb_{\bbP^3}^{3t+1}$, the component $H_0$ itself is of particular
interest. Ellingsrud, Piene and Str\o mme described it in
\cite{EPS:nets_quadrics} as the blow-up of the variety parameterizing
nets of quadrics along a point-plane incidence relation. However,
$H_0$ does not have any known modular interpretation, that is, it does
not satisfy the universal property of a moduli space.

The {\em space of Cohen--Macaulay curves} that H\o nsen introduced 
in \cite{Honsen} gives a different modular compactification $\CM$ of
$\calX$. Instead of adding degenerate schemes as in the
Hilbert scheme case, one considers only curves, that is,
one-dimensional schemes without embedded or isolated points. However,
the curves need not be embedded into $\bbP^3$. Instead they come with a
finite map to $\bbP^3$ that is only generically a closed
immersion. Explicitly, the space $\CM$ parameterizes all pairs $(C,i)$,
where $C$ is a curve and $\func i C \bbP^3$ is a finite map that is an
isomorphism onto its image away from a finite number of closed points
and such that $C$ has Hilbert polynomial $p(t)=3t+1$ with respect to
$i$. The moduli functor $\CM$ is represented by a proper algebraic
space, see \cite{Honsen} and \cite{CM}.

In this work, we describe the points of $\CM$. It turns out that only
two cases can occur. Either the map $i$ is a closed immersion or its
scheme-theoretic image $i(C)$ is a singular plane curve, and $i$
induces an isomorphism away from one singular point $p$ of
$i(C)$. Moreover, there is a bijection between the points of $\CM$ and
the component $H_0$ of the Hilbert scheme of twisted cubics such that
a pair $(C,i)$ where $i$ is not a closed immersion corresponds to the
plane image $i(C)$ augmented with an embedded point at $p$. This
bijection actually defines a birational map between the
spaces. Knowing the points of $\CM$, we can moreover show that the
space is smooth.

We believe that the space $\CM$ actually is isomorphic to the Hilbert
scheme component, giving a modular interpretation for
$H_0$. However, this will have to be shown in future work.

\subsection*{Acknowledgments} I thank Aise Johan de Jong for useful
conversations that partly took place during a visit to Columbia
University financed by SVeFUM, Stiftelsen f\"or Vetenskaplig Forskning
och Utbildning i Matematik.
 
\subsubsection*{Notation and conventions}
Throughout this paper, let $k$ be an algebraically closed field of
characteristic $\opn{char}(k)\neq 2,3$. Unless otherwise stated, the
projective space $\bbP^3$ has coordinates $x,y,z,w$. Moreover, we
write $k[\e]=k[t]/(t^2)$ for the Artin ring of dual numbers. All
schemes considered here are locally Noetherian.

\section{The space of Cohen--Macaulay curves}
For a polynomial $p(t)=at+b\in \bbZ[t]$, let
$\CM_{\bbP^n}^{p(t)}$ be the functor
$\func{\CM_{\bbP^n}^{p(t)}}{(\mathbf{Sch}/k)^\circ}{\mathbf{Sets}}$
that for every $k$-scheme $S$ parameterizes all equivalence classes of
pairs $(C,i)$, where $C$ is a flat scheme over $S$, and $\func i C
\bbP^n_S$ is a finite $S$-morphism such that for every $s\in S$ we
have that
\begin{enumerate}
\item the fiber $C_s$ is Cohen--Macaulay and of pure dimension $1$,
\item the map $\func{i_s}{C_s}{\bbP^n_{\kappa(s)}}$ is an isomorphism
  onto its image away from finitely many closed points,
\item the coherent sheaf $(i_s)_*\calO_{C_s}=(i_*\calO_C)_s$ on
  $\bbP^n_{\kappa(s)}$ has Hilbert polynomial $p(t)$.
\end{enumerate}
Two pairs $(C_1,i_1)$ and $(C_2,i_2)$ in $\CM(C)$ are equal if there exists
an isomorphism $\func \alpha {C_1}{C_2}$ such that $i_2\circ
\alpha=i_1$. 
\begin{thm}[{\cite{Honsen, CM}}]
  The functor $\CM_{\bbP^n}^{p(t)}$ is represented by a proper
  algebraic space.
\end{thm}

In the special case $n=3$ and $p(t)=3t+1$, we write $\mathit{CM}$
instead of $\mathit{CM}_{\bbP^3}^{3t+1}$.

\section{\texorpdfstring{The points of $\mathit{CM}$}{The points of CM}}
In this section, we classify the points $(C,i)$ in $\CM(\spec(k))$
according to the scheme-theoretic image $i(C)$. Moreover, we present
in Subsection~\ref{subsec:specializations} some specialization
relations between them.

\subsection{The scheme-theoretic image} 
We start by giving a description of the curves in $\bbP^3_k$ that can
occur as the scheme-theoretic image of a point $(C,i)\in
\CM(\spec(k))$.
\begin{prop}\label{prop:image_cases}
  Let $(C,i)$ be a $k$-rational point of $\mathit{CM}$. Then one of the
  following two cases occurs.
  \begin{enumerate}
  \item The morphism $i$ is a closed immersion, and the embedded curve
    corresponds to a point on the Hilbert scheme $\Hilb_{\bbP^3}^{3t+1}$ of
    twisted cubics.
  \item The scheme-theoretic image $i(C)$ is a plane curve of degree
    $3$, and $i$ induces an isomorphism onto the image away from one
    closed point in $i(C)$.
  \end{enumerate}
\end{prop}
\begin{proof}
  The finite morphism $i$ factors through the scheme-theoretic image
  $i(C)\subset \bbP^3_k$, and we have an induced short exact
  sequence \[\xymatrix{0\ar[r] & \calO_{i(C)}\ar[r]&i_*\calO_C \ar[r]
    &\calK \ar[r] &0 }\] of coherent $\calO_{\bbP^3_k}$-modules, where
  the cokernel $\calK$ is supported on the finitely many closed points
  where $i(C)$ is not isomorphic to $C$. The Hilbert polynomial
  $p_\calK(t)$ of $\calK$ is constant, equal to a nonnegative integer
  $l$, and we
  have \[p_{i(C)}(t)=p_{i_*\calO_C}(t)-p_{\calK}(t)=3t+1-l.\] In
  particular, we see that $i(C)\subset \bbP^3_k$ is a curve of degree
  $d=3$. Hence, by \cite[Theorem 3.1]{Hartshorne:genus}, its
  arithmetic genus $g_{i(C)}$ is bounded from above by $g_{i(C)}\leq
  \frac 1 2 (d-1)(d-2)=1$. As also $g_{i(C)}=l\geq 0$, it follows that
  there are only two possibilities, namely $l=0$ and $l=1$.

  Suppose first that $l=0$. Then $\calK=0$ and $i$ induces an
  isomorphism between $C$ and $i(C)$, that is, the map $i$ is a
  closed immersion.

  If $l=1$, then the scheme-theoretic image $i(C)$ is a curve of
  degree $d=3$ and genus $g_{i(C)}=1=\frac 1 2 (d-1)(d-2)$. Again by
  \cite[Theorem~3.1]{Hartshorne:genus}, it follows that the curve
  $i(C)$ lies in a plane and does not have any embedded or isolated
  points. Moreover, $p_\calK(t)=1$ implies that the non-isomorphism
  locus consists of a single point in $i(C)$.
\end{proof}

Furthermore, we can show that the non-isomorphism locus is contained
in the singular locus of the scheme-theoretic image $i(C)$.

\begin{lm}\label{lm:normalization}
  Let $(C,i)\in \mathit{CM}_{\bbP^n}^{at+b}(\spec(k))$ with
  scheme-theoretic image $i(C)$, and let $U\subseteq i(C)$ be an
  reduced open subscheme. Then the normalization $\func \nu
  {\widetilde{U}}{U}$ factors through the restriction $\func {i_U}
  {i^{-1}(U)} U$.
\end{lm}
\begin{proof}
 Observe that the morphism $i_U$ is integral and
  birational. Then the statement is a special case of \cite[Tag
  035Q]{stacks-project}.
\end{proof}
\begin{prop}\label{prop:non-iso_sing}
  Let $(C,i)\in \mathit{CM}_{\bbP^n}^{at+b}(\spec(k))$. Then the
  zero-di\-men\-sional locus $Y\subset i(C)$ where $C$ and $i(C)$ are
  not isomorphic is contained in the singular locus of $i(C)$. In
  particular, if the scheme-theoretic image $i(C)$ is smooth, then $i$
  is a closed immersion.
\end{prop}
\begin{proof}
  Let $U=\spec(A)\subset i(C)$ be an open affine subscheme contained
  in the regular locus of $i(C)$, and let $i^{-1}(U)=\spec(B)$. As
  $\widetilde U=U$, the factorization of Lemma~\ref{lm:normalization}
  induces a sequence of injective maps $A\hookrightarrow
  B\hookrightarrow A$. It follows that $i$ induces an isomorphism
  between $U$ and $i^{-1}(U)$, and the non-isomorphism locus $Y$ is
  contained in the singular locus of $i(C)$.

  In particular, the locus $Y$ is empty if $i(C)$ is smooth, that is,
  $i$ is a closed immersion.
\end{proof}

This allows us to give a complete list of the possibilities for the points of
the Cohen--Macaulay space of twisted cubics $\CM$. 

\begin{prop}\label{prop:image_list}
  Let $(C,i)\in \mathit{CM}(\spec(k))$ be such that the map $i$ is not
  a closed immersion. Then the scheme-theoretic image $i(C)$ is a
  plane curve of degree $3$ and $i$ induces an isomorphism between $C$
  and $i(C)$ away from one singular point $p\in i(C)$. Moreover,
  $i(C)$ and $p$ have to be as in one of the following cases:
  \begin{enumerate}
    [label=\normalfont{(\Roman*)}]
  \item\label{item:node} a plane nodal curve, and $p$ is the singular
    point,
  \item\label{item:cusp} a plane cuspidal curve, and $p$ is the
    singular point,
  \item\label{item:quadr+line} a plane conic intersecting a line
    twice, and $p$ is one of the intersection points,
  \item\label{item:quadr+tanget} a plane conic with a tangent line
    through $p$ that lies in its plane,
  \item\label{item:triangle} three coplanar lines
    with three different points of pairwise intersection, and $p$ is
    one of these intersection points,
  \item\label{item:star} three coplanar lines with
    one common point of intersection $p$,
  \item\label{item:double+line1} a plane double line meeting a line in
    its plane, and $p$ is a point on the double line other than the
    intersection point,
  \item\label{item:double+line2} a double line meeting a line as in
    \ref{item:double+line1}, and $p$ is the point of intersection,
  \item\label{item:triple} a planar triple line, and $p$ is any point
    on it.
\end{enumerate}
\end{prop}
The curves listed above are displayed in Figure~\ref{fig:list}.
\begin{proof}
  We showed in Proposition~\ref{prop:image_cases} that $i(C)$ is a
  plane curve of degree $3$ and that the non-isomorphism locus is one
  closed point $p$ in $i(C)$. Moreover, it follows from
  Proposition~\ref{prop:non-iso_sing} that $p$ is a singular point.

  The list consists of all types, up to projective equivalence, of
  singular plane curves of degree $3$ and the possibilities of
  choosing a singular point on it.
\end{proof}

\begin{figure}[h]
    \centering
    \includegraphics[width=0.85\textwidth]{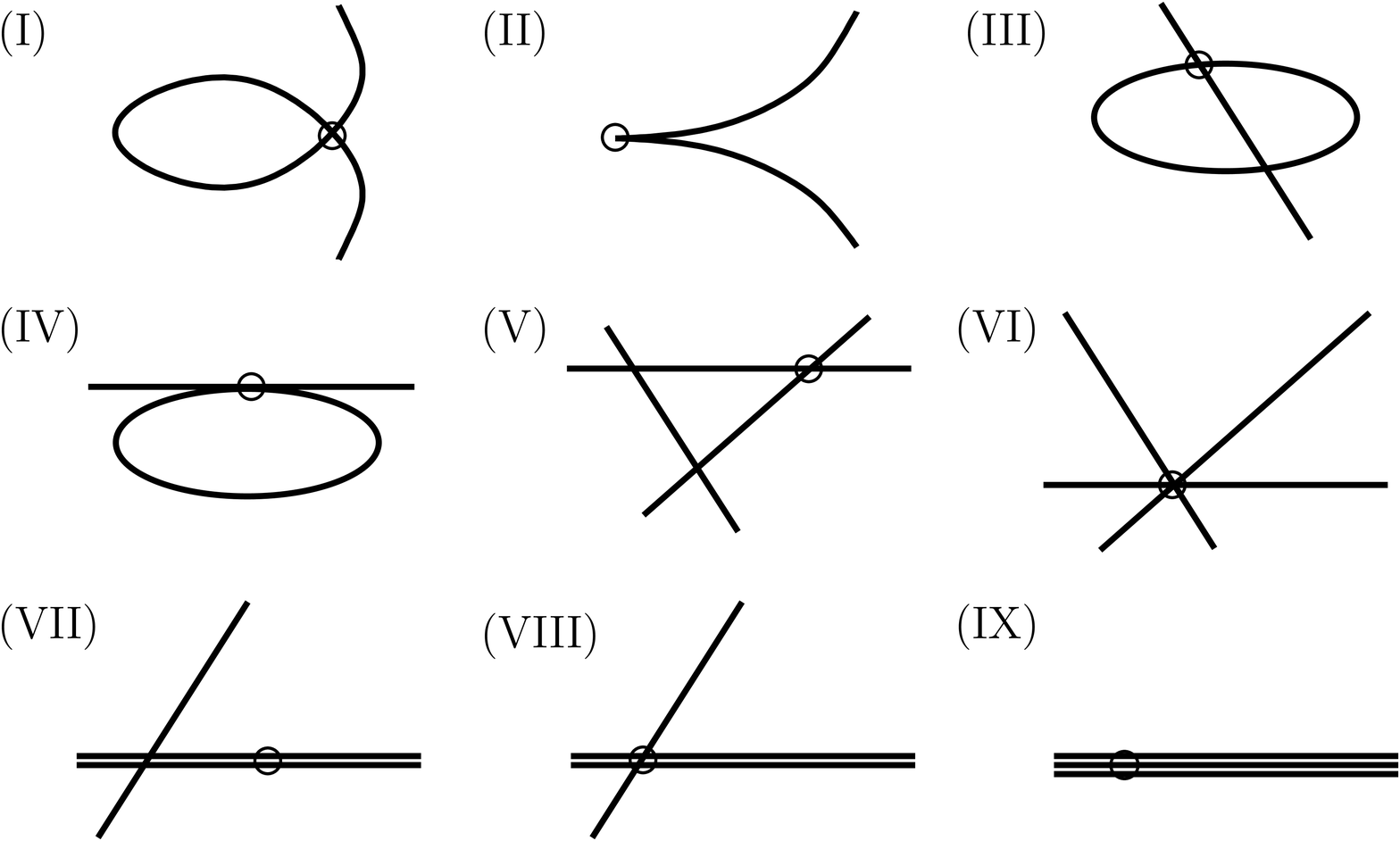}
    \caption{The possible scheme-theoretic images and the
      non-isomorphism point.}
    \label{fig:list}
  \end{figure}
\subsection{Existence}
All the curves listed above actually occur as scheme-theoretic images,
that is, for every choice of plane curve $D$ and singular point
$p$ as in Proposition~\ref{prop:image_list}, there exists at least one
point $(C,i)$ in $\CM(\spec(k))$ such that $i(C)=D$ and $p$ is the
non-isomor\-phism locus.

\begin{thm}\label{thm:existence}
  For every plane cubic $D\subset \bbP^3_k$ of degree $3$ with
  singular point $p\in D$, there exists $(C,i)\in
  \CM(\spec(k))$ with the following properties:
  \begin{enumerate}
  \item The scheme-theoretic image of $C$ in $\bbP^3_k$ is $D$, and the
    induced map $C\to D$ is an isomorphism away from $p$.
  \item The curve $C$ is the flat degeneration of a twisted cubic, and
    it has an embedding $h\mathpunct : C\subset \bbP^3_k$ such that
    $i^*\calO_{\bbP^3_k}(1)=h^*\calO_{\bbP^3_k}(1)$.
  \end{enumerate}
\end{thm}
\begin{proof}
  Without loss of generality, we can assume that the curve $D$ is
  contained in the plane $z=0$ and that it is given by a cubic form
  $q(x,y,w)$ with a singularity at the point $p=[0 : 0:0:1]$. For
  every type of curve as in Proposition~\ref{prop:image_list}, it
  suffices to consider one particular example for $q(x,y,w)$ as they
  all are projectively equivalent.

  In all cases, the curve is given as $C=\Proj(k[x,y,w,u]/I)$ for some
  ideal $I$, and the morphism $\func i C \bbP^3_k$ is induced by the
  homomorphism of graded rings
  $\func{\varphi}{k[x,y,z,w]}{k[x,y,w,u]/I}$ with $\varphi(x)=x$,
  \mbox{$\varphi(y)=y$}, $\varphi(z)=0$ and $\varphi(w)=w$.
  \begin{enumerate}[label={\normalfont(\Roman*)}]
  \item With $I=(xu-yw,yu-x(x+w), u^2-w(x+w))$, the curve $C$ is a
    twisted cubic, and the scheme-theoretic image $i(C)$ is the plane
    nodal curve defined by the ideal
    $\ker(\varphi)=(z,x^3+x^2w-y^2w)$. Note moreover that $i$ is an
    isomorphism onto the image away from the node $p=[0:0:0:1]$.
  \item Let $I=(xu-yw, yu-x^2,u^2-xw)$. Then $C$ is a twisted cubic,
    and the scheme-theoretic image is the plane cuspidal curve defined
    by the ideal $\ker(\varphi)=(z,x^3-y^2w)$.
  \item The conic intersecting a line not in its plane having the ideal
    $I=(xu,yu-(x^2+yw),u^2-uw)$, has scheme-theoretic image given by
    $\ker(\varphi)=(z,x^3+xyw)$, that is, a conic meeting a line in
    two points, one of them being the non-isomorphism point
    $p=[0:0:0:1]$.
  \item With $I=(xu-(x^2+yw), yu, u^2-(x^2+yw))$, the curve $C$ is a
    conic intersecting a line that does not lie in its plane, and the
    image is the conic with tangent line, given by the ideal
    $\ker(\varphi)=(z,x^2y+y^2w)$.
  \item For $I=(xu,yu-yw,u^2-uw)$, the curve $C$ consists of three
    noncoplanar lines with two intersection points. The
    scheme-theoretic image, given by the ideal $\ker(\varphi)=(z,xyw)$,
    is three coplanar lines such that two of them intersect in the
    non-isomorphism locus $p=[0:0:0:1]$.
  \item With $I=(xu-xy,yu-xy,u^2-yu)$, we have that $C$ consists of
    three concurrent but not coplanar lines. The scheme-theoretic
    image is three concurrent and coplanar
    lines, and it is given by the ideal $\ker(z,x^2y-xy^2)$.
  \item For $I=(xu,yu-xw,u^2)$, the curve $C$ is a double line of
    genus $-1$ meeting a line. The image is a planar double line and a
    line in its plane, given by the ideal
    $\ker(\varphi)=(z,x^2w)$. The curves $C$ and $i(C)$ are isomorphic
    away from the point $p=[0:0:0:1]$ that lies on the double line but
    is not the intersection point.
  \item Similarly, the ideal $I=(xu-x^2,yu,u^2-xu)$ describing a
    planar double line and a line not in its plane gives as image the
    planar double line and the line in its plane defined by the ideal
    $\ker(\varphi)=(z,x^2y)$. In this case the non-isomorphism point
    is the intersection point $p=[0:0:0:1]$.
  \item Finally, if $C$ is the nonplanar triple line defined by the
    ideal $I=(xu,yu-x^2,u^2)$, the image is the planar triple line
    given by the ideal $(z,x^3)$.
  \end{enumerate}
  In all cases the curve $C$ is given with an embedding into the
  projective space $\bbP^3_k=\Proj(k[x,y,w,u])$ so that
  $i^*\calO_{\bbP^3_k}(1)=\calO_C(1)$. Moreover, every curve is the
  specialization of a twisted cubic.
\end{proof}
\vspace{\baselineskip}

\begin{rem}
   Let for example $D$ be the plane curve given by the ideal
  $(z,x^3+x^2w-y^2w)$ as in case ~\ref{item:node}. Then we consider
  the flat one-parameter family $Z\subset \bbP^3_{k[t]}$ generated by
  the homogeneous polynomials $f_1=xz-tyw$, $f_2=yz-tx(x+w)$,
  $f_3=z^2-t^2w(x+w)$ and $q=x^3+x^2w-y^2w$ in $k[t][x,y,z,w]$. Note
  that $yf_1-xf_2=tq$. For $t\neq 0$ the fiber $Z_t$ is a twisted
  cubic, whereas $Z_0$ is the plane nodal curve $D$ with an embedded
  point at the singularity given by the ideal $(xz,yz,z^2,q)$. Then
  the ideal $I$ of the curve $C$ is generated by the polynomials
  $g_1,g_2,g_3\in k[x,y,w,u]$ that are obtained by dividing $f_1$ and
  $f_2$ by $t$ and $f_3$ by $t^2$ and setting $u=t^{-1}z$.

  More generally, all curves and maps in the proof of
  Theorem~\ref{thm:existence} were constructed in a similar way: We
  consider a flat one-parameter family $Z\subset \bbP^3_{k[t]}$ such
  that the fiber $Z_t$ is a Cohen--Macaulay curve with Hilbert
  polynomial $p(n)=3n+1$ for $t\neq 0$, and $Z_0$ is the plane curve
  $D$ with an embedded point supported at $p$. Suitable generators
  $f_1, f_2, f_3,q\in k[t][x,y,z,w]$ of the ideal defining $Z$ give
  then rise to the generators $g_1, g_2,g_3\in k[x,y,w,u]$ of $I$.
\end{rem}

\subsection{Uniqueness}\label{subsec:uniqueness}
In the next step, we show that the curves constructed in the proof of
Theorem~\ref{thm:existence} are the unique solutions, see
Theorem~\ref{thm:uniqueness}.
\begin{lm}\label{lm:property_extension}
  Let $(C,i)\in \mathit{CM}(\spec(k))$ be such that the map $i$ is not
  a closed immersion. Assume that the scheme-theoretic image $i(C)$ is
  contained in the plane $z=0$ and that $i$ induces an isomorphism
  between the curve $C$ and the image away from the singular point
  $[0:0:0:1]$ on $i(C)$.  Let further $A$ be the $k$-algebra such that
  $i(C)\cap D_+(w)=\spec(A)$, and let
  $i^{-1}(\spec(A))=\spec(B)$. Then the map $i$ corresponds to an
  inclusion $A\subset B$ of rings such that
  \begin{enumerate}
  \item\label{item:image_1} $\dim_k(B/A)=1$,
  \item\label{item:image_2} $x B \subseteq A$ and $y B\subseteq A$, and
  \item\label{item:image_3} if $a\in A$ is not a zero divisor in $A$,
    then $a$ is not a zero divisor in $B$.
  \end{enumerate}
\end{lm}
\begin{proof}
  The first property \ref{item:image_1} follows directly since the
  Hilbert polynomials of $C$ and $i(C)$ differ by 1 and the
  non-isomorphism locus is contained in $\spec(A)$.

  Note that the quotient $B/A$ is only supported at the maximal ideal
  $\m=(x,y)$ of $A$. By property~\ref{item:image_1}, it follows that
  $\Ann_A(B/A)=\m$ and property \ref{item:image_2} holds.
 
  For property \ref{item:image_3}, assume that $a$ is a zero divisor
  of $B$, that is, that $a$ is contained in an associated prime ideal
  $\p$ of $B$. As the curve $C$ is Cohen--Macaulay without isolated
  points, it follows that $\p$ is a minimal prime ideal that is
  not maximal. The restriction $\p\cap A$ is then a minimal prime
  ideal in $A$ that contains $a$. This implies that $a$ is a zero
  divisor.
\end{proof}

\begin{thm}\label{thm:uniqueness}
  For every plane cubic curve $D\subset \bbP^3_k$ of degree $3$ with
  singular point $p\in D$, there exists at most one $k$-rational point
  $(C,i)$ on $\CM$ such that $i(C)=D$ and the induced
  map $C\to D$ is an isomorphism away from $p$.
\end{thm}
\begin{proof}
  We prove the statement individually for the different possibilities
  of $D$ as listed in Proposition~\ref{prop:image_list}.  In the cases
  \ref{item:node} of a nodal curve, \ref{item:cusp}~of a cuspidal
  curve, \ref{item:quadr+line} of a conic and a line intersecting
  twice and \ref{item:triangle} of three coplanar lines, the point $p$
  is an isolated singular point. Lemma~\ref{lm:normalization} and
  comparison of the Hilbert polynomials imply that locally around $p$
  the map $C\ra D$ has to be the normalization, and hence it is
  unique.

  In the remaining cases, we can without loss of generality assume
  that $D$ is contained in the plane $z=0$ and that $p=[0:0:0:1]$,
  that is, $D$ is given by an ideal $I=(z,q(x,y,w))$, where
  $q(x,y,w)$ is a cubic form with singularity at $p$. Then we show
  that for \mbox{$D\cap D_+(w)=\spec(A)$} there exists, up to
  $A$-algebra isomorphism, only one $k$-algebra extension $A\subset B$
  satisfying the properties of Lemma~\ref{lm:property_extension}.

  In case \ref{item:quadr+tanget}, the curve $D$ consists of a
  conic and a tangent line, say $q(x,y,w)=x^2y+y^2w$. Note that all
  such curves are projectively equivalent, and hence it suffices to
  show the claim for one particular choice of cubic form $q(x,y,w)$.
  In the ring $A=k[x,y]/(x^2y+y^2)$ we have that $y^{n}=y(-x^2)^{n-1}$
  for every $n\in\bbN$. In particular, every element $a\in A$ can be
  written uniquely as $a=f(x)+yg(x)$ with $f(x),g(x)\in k[x]$. Now let
  $A\subset B$ be as above, and let $b\in B\setminus A$. As $xb,yb\in
  A$ and $y(xb)=x(yb)$, one can show that there are polynomials
  $g_1(x),g_2(x)\in k[x]$ such that \[\begin{cases} xb=
    xg_2(x)+(x^2+y)g_1(x) \\ yb = yg_2(x). \end{cases}\] We can write
  $g_1(x)=c+xu(x)$ for $c\in k$ and $u(x)\in k[x]$. Replacing $b$ by
  $b-g_2(x)-(x^2+y)u(x)$, we get that \[\begin{cases}
    xb=c(x^2+y)\\yb=0.\end{cases}\] Moreover, it follows that
  $xb^2=xc^2(x^2+y)$. As $x$ is not a zero divisor in $B$ by property
  \ref{item:image_3} in Lemma~\ref{lm:property_extension}, we can
  conclude that $b^2=c^2(x^2+y)$ and $c\neq 0$. After replacing $b$ by
  $c^{-1}b$, we can consider the $A$-algebra
  $B':=A[b]/(xb-(x^2+y),yb,b^2-(x^2+y))$ that lies between $A$ and
  $B$. As $\dim_k(B'/A)= 1 = \dim_k(B/A)$, it follows that $B\cong
  B'$.

  The cases \ref{item:star} to \ref{item:triple} are shown in the same
  way. For \ref{item:star}, the plane curve $D$ consists of three
  concurrent lines, and we can assume that $q(x,y,w)=x^2y-xy^2$ and
  get $B\cong A[b]/(xb-xy,yb-xy,b^2-xy)$. If the scheme-theoretic
  image is given by $q(x,y,w)=x^2w$, as in the situation of
  \ref{item:double+line1}, the extension is $B\cong
  A[b]/(xb,yb-x,b^2)$. If $p$ is the intersection point of a double
  line and a line as in \ref{item:double+line2}, we can assume that
  $q(x,y,w)=x^2y$ and get $B\cong A[b]/(xb-x^2,yb,b^2-x^2)$. In the
  last case \ref{item:triple}, the curve $D$ is the triple line given
  by $q(x,y,w)=x^3$. We show then that $B\cong A[b]/(xb,yb-x^2,b^2)$.
\end{proof}
\begin{rem}
  Note that the extensions $B$ constructed in the proof are affine
  charts of the curves $C$ listed in the proof of
  Theorem~\ref{thm:existence}.
\end{rem}

\subsection{\texorpdfstring{Classification of the points of
    $\boldsymbol{\CM}$}{Classification of the points of CM}}
We summarize the results of the previous subsections as follows.
\begin{thm}\label{thm:description_CM}
  There is a one-to-one correspondence between the $k$-rational points
  of $\CM$ and the union of the set of equidimensional Cohen--Macaulay
  curves in $\bbP^3_k$ with Hilbert polynomial $3t+1$ and the set of
  singular plane curves in $\bbP^3_k$ together with a singular point $p$
  on it.
\end{thm}
\begin{proof}
  We have seen in Proposition~\ref{prop:image_cases} that for every pair
  $(C,i)$ in $\CM(\spec(k))$, the map $i$ is either a closed
  immersion or an isomorphism onto a plane curve away from one point
  $p$ that has to be singular by
  Proposition~\ref{prop:image_list}. 

  Conversely, every embedding of an equidimensional Cohen--Macaulay
  curve with Hilbert polynomial $3t+1$ gives a point on
  $\CM$. Moreover, we have seen in Theorem~\ref{thm:existence} and
  Theorem~\ref{thm:uniqueness} that for every plane curve $D$ with
  singular point $p$ the exists a unique point $(C,i)$ on $\CM$ such
  that $i$ induces an isomorphism between $C$ and $D$ away from $p$.
\end{proof}

\subsection{\texorpdfstring{Specializations in
    $\boldsymbol{\CM}$}{Specializations in
    CM}} \label{subsec:specializations}

Comparing the ideals in the proof of Theorem~\ref{thm:existence}, we
can see that all points of $CM$ specialize to a point corresponding to
a pair $(C,i)$ where the scheme-theoretic image is a triple line.

\begin{ex}
  Let $(C,i)\in\CM(\spec(k[t]))$ be a family of Cohen--Macaulay curves
  where $C\subset \bbP^3_{k[t]}=\proj(k[t][x,y,w,u])$ is given by the
  ideal $I=(xu,yu-x(x+ty),u^2)$, and the map $i$ corresponds to the
  homomorphism of graded
  rings \[\func{\varphi}{k[t][x,y,z,w]}k[t][x,y,w,u]/I\] given by
  $\varphi(x)=x$, $\varphi(y)=y$, $\varphi(z)=0$ and $\varphi(w)=w$.

  For $t\neq 0$, the scheme-theoretic image $i_t(C_t)$ consists
  of the double line intersecting a line $(z,x^3+tx^2y)$, and $i_t$
  induces an isomorphism away from the intersection point.

  The scheme-theoretic image $i_0(C_0)$, on the contrary, is the plane
  triple line $(z,x^3)$.
\end{ex}

In a similar way, we can show that all types~\ref{item:node} to
\ref{item:double+line2} specialize to the case of a triple line
\ref{item:triple}. Specifically, we have the chart of specializations
as shown in Figure~\ref{fig:curves}.
\begin{figure}[ht]
 \includegraphics[width=0.75\textwidth]{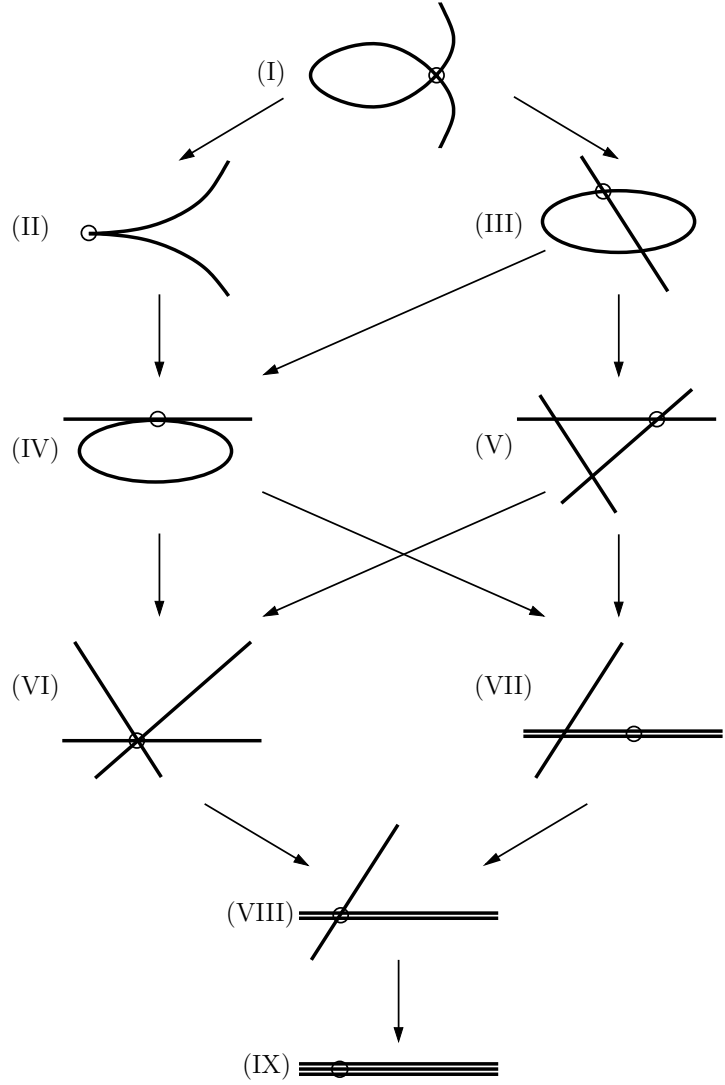}
 \caption{Specializations between points of $\CM$ with
   scheme-theoretic image and non-isomorphism locus of types
   \ref{item:node} to \ref{item:triple} as in
   Proposition~\ref{prop:image_list}.}
\label{fig:curves}
\end{figure}

A similar diagram of specializations for the component $H_0$ of the
Hilbert schemes of twisted cubics can be found in
\cite[p.~40]{Harris:curves}.

\section{The Hilbert scheme of twisted cubics}
Knowing the points of $\CM$, we can now establish a bijection with the
points of one component of the Hilbert scheme of twisted cubics.
\begin{thm}[{\cite{PS:twistedcubic}}]
  \label{thm:Hilbert_scheme}
  The Hilbert scheme $\Hilb_{\bbP^3}^{3t+1}$ consists of two
  components $H_0$ and $H_1$. The points of $H_0$ are the
  degenerations of a twisted cubic, namely all equidimensional
  Cohen--Macaulay curves in $\bbP^3$ with Hilbert polynomial $3t+1$
  and all singular, plane curves with an embedded point that is supported at a
  singularity and emerges from the plane. The component $H_0$ is
  smooth and has dimension~$12$.
\end{thm}

\begin{prop}\label{prop:comparison}
  There is a bijection between the set of $k$-rational points of the
  space of Cohen--Macaulay curves $\CM$ and the
  set of $k$-rational points of the component $H_0$ of the Hilbert
  scheme $\Hilb_{\bbP^3}^{3t+1}$. Moreover, the open subfunctor $U$ of
  $CM$ corresponding to closed immersions is isomorphic to the open
  subscheme of $H_0$ corresponding to Cohen--Macaulay curves.
\end{prop}
\begin{proof}
  The locus $U$ in $\CM$ coincides with the Cohen--Macaulay locus in
  $\Hilb_{\bbP^3}^{3t+1}$. Moreover, every equidimensional
  Cohen--Macaulay curve in $\bbP^3$ corresponds to a point of $H_0$,
  see the proof of \cite[Lemma~1]{PS:twistedcubic}. The remaining
  points in both $\CM(\spec(k))$ and $H_0(\spec(k))$ are in bijection
  with the set of pairs consisting of a plane curve of degree $3$ in
  $\bbP^3$ and a singular point on it, see
  Theorem~\ref{thm:description_CM} and
  Theorem~\ref{thm:Hilbert_scheme}.
\end{proof}
\begin{cor}\label{cor:irred_12}
  The space of Cohen--Macaulay curves $\CM$ has an irreducible open
  dense subscheme that is smooth of dimension $12$. In
  particular, $\CM$ is irreducible and has dimension~$12$. 
\end{cor}
\begin{proof}
  The subspace $U$ in Proposition~\ref{prop:comparison} has the
  required properties. This implies that $\CM$ itself is irreducible
  and has dimension ~$12$.
\end{proof}

\section{Deformations}\label{sec:deformations}
The goal of this section is to show that $\CM$ is smooth. In
particular we compute the dimension of the tangent space of $\CM$ at
one certain point.

A {\em first-order deformation} of a point $(C,i)\in
\CM_{\bbP^n}^{at+b}(\spec(k))$ is an element $(\widetilde C, \tilde i)\in
\CM_{\bbP^n}^{at+b}(\spec(k[\e])$ such that the
diagram \[\xymatrix{C\ar[r]^i \ar@{^(->}[d]& \bbP^n_k \ar@{^(->}[d] \\
  \widetilde C\ar[r]^{\tilde i} & \bbP^n_{k[\e]}}\] is Cartesian. The
space of these first-order deformations is isomorphic to the tangent
space of $\CM_{\bbP^n}^{at+b}$ at the point $(C,i)$.

We show first that the curve $C$ can be embedded in a projective space
$\bbP^N_k$ in such a way that the scheme $\widetilde C$ is given as
deformation of $C$ in $\bbP^N_k$.

\begin{prop}[{\cite[Proposition 4.14]{CM}}]\label{prop:embedding}
  There exist $m,N\in \bbN$ such that for every field $k$ and every
  $(C,i)\in \CM_{\bbP^n}^{at+b}(\spec(k))$ there exists a closed
  immersion $\injfunc j C{\bbP^N_k}$ such that
  $j^*\calO_{\bbP^N_k}(1)=i^*\calO_{\bbP^n_k}(m)$ and
  $j^*x_0,\ldots,j^*x_N$ form a basis of
  $H^0(C,i^*\calO_{\bbP^n_k}(m))$.
\end{prop}
\begin{prop}\label{prop:deformation_subscheme}
  Let $(\widetilde{C}',\tilde{i}')\in
  \CM_{\bbP^n}^{at+b}(\spec(k[\e]))$ be a first-order deformation of
  the point $(C,i)\in \CM_{\bbP^n}^{at+b}(\spec(k))$. Suppose that the
  curve $C$ is given as a closed subscheme $j\mathpunct: C\subset
  \bbP^N_k$ as in Proposition~\ref{prop:embedding}. Then $(\widetilde
  C',\tilde i')=(\widetilde C,\tilde i)$ in
  $\CM_{\bbP^n}^{at+b}(\spec(k[\e]))$ for a first-order deformation
  $\widetilde C$ of the closed subscheme $C$ of $\bbP^N_k$.
\end{prop}
\begin{proof}
  Let $m,N\in \bbN$ be as in Proposition~\ref{prop:embedding}. By
  \cite[Proposition~4.16]{CM}, the $k[\e]$-module $H^0(\widetilde
  C',(\tilde i')^*\calO_{\bbP^n_{k[\e]}}(m))$ is free of rank $N+1$
  and \[H^0(C,i^*\calO_{\bbP^n_k}(m)) = H^0(\widetilde C',(\tilde
  i')^*\calO_{\bbP^n_{k[\e]}}(m))\otimes_{k[\e]} k[\e]/(\e).\]
  Therefore we can choose a basis $\tilde s_0,\ldots,\tilde s_N$ of
  $H^0(\widetilde C',(\tilde i')^*\calO_{\bbP^n_{k[\e]}}(m))$ that
  lifts the basis $j^*x_0,\ldots,j^* x_N$ of
  $H^0(C,i^*\calO_{\bbP^n_k}(m))$. Then, by
  \cite[Proposition~4.19]{CM}, the choice of global sections induces a
  closed immersion $\injfunc{\tilde j}{\widetilde
    C'}{\bbP^N_{k[\e]}}$. Note that the commutative
  diagram \[\xymatrix{
    \bbP^N_k \ar@{^(->}[d]&\ar@{}[l]|-*{\supset} C\ar@{^(->}[d] \\
    \bbP^N_{k[\e]}& \ar@{_(->}[l]_-{\tilde j} \widetilde C' }\] is
  Cartesian. Now let $\widetilde C$ be the scheme-theoretic image of
  the closed immersion $\tilde j$ and $\iso {\tilde \alpha}
  {\widetilde C}{\widetilde C'}$ the induced isomorphism. Then
  $\widetilde C\subset \bbP^N_{k[\e]}$ is flat over $\spec(k[\e])$,
  and its restriction modulo $\e$ is $C$. Hence $\tilde C$ is a
  first-order deformation of $C\subset \bbP^3_k$.  The restriction
  $\alpha$ of $\tilde \alpha$ modulo $\e$ is an automorphism of $C$
  such that $i\circ \alpha=i$. Then, by \cite[Theorem~2.19]{CM}, the
  map $\alpha$ is the identity. With $\tilde i:=\tilde i'\circ \tilde
  \alpha$, it follows that the restriction of $\tilde i$ modulo $\e$
  is $i$ and that $(\widetilde C',\tilde i')=(\widetilde C,\tilde i)$
  in $\CM_{\bbP^n}^{at+b}(\spec(k[\e]))$.
\end{proof}

From now on we treat the special case $n=3$ and $p(t)=3t+1$. We show
that $\mathit{CM}$ is smooth by proving that the tangent space at
every point has dimension 12. In Section~\ref{subsec:specializations},
we have seen that all maps $\func i C \bbP^3_k$ specialize to a map
such that the scheme-theoretic image is a plane triple line. Hence it
suffices to study the $k[\varepsilon]$-deformations at such a point of
$\CM$.
\begin{lm}\label{lm:twisted_cubic_emb}
  Let $(C,i)\in \CM(\spec(k))$ be a point of $\CM$. Then the following
  holds.
  \begin{enumerate}
  \item\label{item:twisted_cubic_emb1} The coherent sheaf $i_*\calO_C$
    is $1$-regular.
  \item\label{item:twisted_cubic_emb2} $h^0(C,i^*\calO_{\bbP^3_k}(1))=4$.
  \item\label{item:twisted_cubic_emb3} The global sections of
    $i^*\calO_{\bbP^3_k}(1)$ separate points and tangent vectors.
  \item\label{item:twisted_cubic_emb4} Every choice of basis of
    $H^0(C,i^*\calO_{\bbP^3_k}(1))$ gives a closed immersion $\injfunc j
    C \bbP^3_k$.
  \end{enumerate}
\end{lm}
\begin{proof}
  We have seen in Theorem~\ref{thm:existence} and~\ref{thm:uniqueness}
  that the curve $C$ is a Cohen--Macaulay specialization of a twisted
  cubic and that it has an embedding $\injfunc h C \bbP^3_k$ such that
  $i^*\calO_{\bbP^3_k}(1) = h^*\calO_{\bbP^3_k}(1)$. Let $\calI$ be
  the sheaf of ideals describing $C$ as a subscheme in
  $\bbP^3_k$. Then by \cite[Exemple 1]{Ellingsrud} there exists a
  short exact sequence \[\xymatrix{0\ar[r] &
    \calO_{\bbP^3_k}(-3)^{\oplus 2} \ar[r] &
    \calO_{\bbP^3_k}(-2)^{\oplus 3} \ar[r] & \calI \ar[r] & 0.}\] From
  the corresponding long exact sequence in cohomology we conclude that
  $H^r(\calI(d))=0$ for all $d$ and $r=1$ and $r\geq 4$. Moreover, we
  get that $H^0(\calI(d))=0$ for $d<2$, $H^2(\calI(d))=0$ for $d\geq
  0$ and $H^3(\calI(d))=0$ for $d\geq -1$. In particular, it follows
  that $\calI$ is $2$-regular. Applying these results on the
  cohomology of $\calI$ to the short exact sequence \[\xymatrix{0
    \ar[r] & \calI \ar[r] & \calO_{\bbP^3_k} \ar[r] & h_*\calO_C\ar[r]
    &0,}\] we conclude that $h^0(\bbP^3_k,(h_*\calO_C)(1)) =
  h^0(\bbP^3_k,\calO_{\bbP^3_k}(1)) = 4$ and that $h_*\calO_C$ is
  $1$-regular. Finally, due to the projection formula and the fact
  that the maps $i$ and $h$ are finite, we have that 
  \begin{align*}
    H^r(\bbP^3_k,(i_*\calO_C)(d)) & = H^r(C,i^*\calO_{\bbP^3_k}(d)) =
    H^r(C,h^*\calO_{\bbP^3_k}(d)) = \\ &=
    H^r(\bbP^3_k,(h_*\calO_C)(d))
  \end{align*}
  for all $d$ and $r\geq 0$, and we have shown properties
  \ref{item:twisted_cubic_emb1} and \ref{item:twisted_cubic_emb2}.

  By \cite[Proposition II.7.3]{Hartshorne}, the global sections of the
  invertible sheaf $i^*\calO_{\bbP^3_k}(1)=h^*\calO_{\bbP^3_k}(1)$
  separate points and tangent vectors, hence we have
  property~\ref{item:twisted_cubic_emb3}. In particular, every basis
  of global sections of $i^*\calO_{\bbP^3_k}(1)$ separates points and
  tangent vectors and induces a closed immersion $\injfunc j C
  \bbP^3_k$, again by \cite[Proposition II.7.3]{Hartshorne}. This
  shows property~\ref{item:twisted_cubic_emb4} and concludes the
  proof.
\end{proof}
In terms of the notation of Proposition~\ref{prop:embedding}, the lemma
says that we have that $m=1$ and $N=3$ in the twisted cubic case.

\begin{prop}\label{prop:triple_smooth}
  Let $(C,i)$ be a point in $\mathit{CM}(\spec(k))$ such that the
  scheme-theoretic image is a plane triple line. Then the tangent
  space at this point has dimension $12$.
\end{prop}
\begin{proof}
  Without loss of generality, we can assume that the curve $C$ is
  given by the ideal $I=(xu,yu-x^2,u^2)$ in
  $\bbP^3_k=\Proj(k[x,y,w,u])$ and that $i$ corresponds to the the
  homomorphism of graded rings
  \[\func{\varphi}{k[x,y,z,w]}{k[x,y,w,u]/I}\] with
  $\varphi(x)=x$, $\varphi(y)=y$, $\varphi(z)=0$ and
  $\varphi(w)=w$. 

  As in Proposition~\ref{prop:deformation_subscheme}, we study first
  the deformations of $C$ as a subscheme of the projective space
  $\bbP^3_k$. These deformations are in one-to-one correspondence with
  the elements of $H^0(C,\calN_{C/\bbP^3_k})$, see for example
  \cite[Theorem~2.4]{Hartshorne:Def}, and we can compute them from the
  exact sequence \[\xymatrix{0\ar[r] & \calN_{C/\bbP^3_k} \ar[r] &
    \calO_C(2)^{\oplus 3}\ar[r] & \calO_C(3)^{\oplus 2}}\] induced by
  a resolution of the ideal $I$. It follows that the space of
  deformations has dimension $12$, and for every
  $\mathbf{a}=(a_1,\ldots,a_{12})\in k^{12}$ we get a deformation
  $\widetilde C_{\mathbf{a}}\subset\bbP^3_{k[\e]}$ defined by the
  ideal $\tilde I_{\mathbf{a}}$ generated by the
  polynomials \begin{align*} p_{1,\mathbf{a}}(x,y,w,u) =&xu +
    \e(a_1x^2 + a_2xy + a_3xw + a_4y^2 + a_5yw+a_6wu), \\
    p_{2,\mathbf{a}}(x,y,w,u) = &yu-x^2 +
    \\& +\e (a_7x^2 + a_8xy + a_9xw + a_{10}y^2 + a_{11}yw + a_{12}wu), \\
    p_{3,\mathbf{a}}(x,y,w,u) = &u^2 + \e ((a_2+a_{10})x^2 + a_4xy +
    a_5xw + (a_3+a_{11})wu).
  \end{align*}
  Every deformation of $(C,i)$ is then given by a map $\tilde
  i_{\mathbf{a},\mathbf{b}}$ associated
  to \[\func{\tilde\varphi_{\mathbf{a},\mathbf{b}}}
  {k[\e][x,y,z,w]}{k[\e][x,y,w,u]/\tilde I_{\mathbf{a}}}\] defined by
  \begin{align*}
   \tilde \varphi_{\mathbf{a},\mathbf{b}}(x) & = x+
    \e(b_1x+b_2y+b_3w+b_4u)\\ \varphi_{\mathbf{a},\mathbf{b}}'(y) &
    =y+\e(b_5x+b_6y+b_7w+b_8u) \\
    \tilde \varphi_{\mathbf{a},\mathbf{b}}(z) & =0+
    \e(b_9x+b_{10}y+b_{11}w+b_{12}u)
    \\\tilde\varphi_{\mathbf{a},\mathbf{b}}(w) & =w+
    \e(b_{13}x+b_{14}y+b_{15}w+b_{16}u)
  \end{align*}
  for $\mathbf{b}=(b_1,\ldots,b_{16})\in k^{16}$.  Thus the pairs
  $(\widetilde C_{\mathbf{a}}, \tilde i_{\mathbf{a},\mathbf{b}})$ give
  all deformations, and the dimension of the tangent space at the
  point $(C,i)$ is at most $12+16=28$.

  Recall that in $\CM$ we only consider isomorphism classes of
  pairs. In particular, we have $(\widetilde C_{\mathbf{a}},\tilde
  i_{\mathbf{a},\mathbf{b}}) = (\tilde C_{\mathbf{a}'},\tilde
  i_{\mathbf{a}',\mathbf{b}'})$ in $\mathit{CM}(\spec(k[\e]))$ for
  $\mathbf{a}, \mathbf{a}'\in k^{12}$ and $\mathbf{b},\mathbf{b}'\in
  k^{16}$ if and only if there exists an isomorphism
  $\iso{\tilde\alpha}{\widetilde C_{\mathbf{a}}}{\widetilde
    C_{\mathbf{a}'}}$ such that $\tilde
  i_{\mathbf{a},\mathbf{b}}=\tilde i_{\mathbf{a}',\mathbf{b}'}\circ
  \tilde \alpha$. As the restriction $\alpha$ of $\tilde\alpha$ to
  $k=k[\e]/(\e)$ is an automorphism of $C$ such that $i=i\circ
  \alpha$, it follows from \cite[Theorem~2.19]{CM} that $\alpha$ is
  the identity morphism.

  We consider the particular case that $\tilde\alpha$ is induced by a
  homomorphism of graded rings
  $\func{\tilde\sigma_{\mathbf{s}}}{k[\e][x,y,w,u]} k[\e][x,y,w,u]$
  with
  \begin{align*}
    \tilde \sigma_{\mathbf{s}}(x) & =x+\e(s_1x+s_2y+s_3w+s_4u) \\
    \tilde\sigma_{\mathbf{s}}(y) & =y+\e(s_5x+s_6y+s_7w+s_8u) \\
    \tilde\sigma_{\mathbf{s}}(w) & =w+\e(s_9x+s_{10}y+s_{11}w+s_{12}u) \\
    \tilde\sigma_{\mathbf{s}}(u) &
    =u+\e(s_{13}x+s_{14}y+s_{15}w+s_{16}u).
  \end{align*}
  for $\mathbf{s}=(s_1,\ldots,s_{16})\in k^{16}$. Then one can compute
  that $\tilde\sigma_{\mathbf{s}}(C_{\mathbf{a}})=C_{\mathbf{a}'}$ and
  $\tilde\sigma_{\mathbf{s}}\circ
  \tilde\varphi_{\mathbf{a},\mathbf{b}} =
  \tilde\varphi_{\mathbf{a}',\mathbf{b}'}$ for some
  $\mathbf{s}=(s_1,\ldots,s_{16})\in k^{16}$ if and only if the
  following $12$ conditions hold:
  \begin{gather*}
    a_2-a_{10} = a_{2}'-a_{10}', \qquad a_3-a_{11} =a_{3}'-a_{11}',
    \qquad a_4=a_4', \qquad a_5 = a_5', \\ b_2+\frac 1 3(a_8-a_1)
    =b_2'+\frac 1 3(a_8'-a_1'), \qquad b_3+\frac 1 2 a_9 = b_3'+\frac
    1 2 a_9', \\ b_4-a_6 = b_4'-a_6',\qquad b_7-a_{12}=b_{7}'-a_{12}' \\
    b_9=b_9',\qquad b_{10} =b_{10}',\qquad b_{11} = b_{11}', \qquad
    b_{12}= b_{12}'.
  \end{gather*}
  So the equivalence class of the element $(\widetilde C_{\mathbf{a}},
  \tilde i_{\mathbf{a}, \mathbf{b}})$ in $\CM(\spec(k[\e]))$ depends
  only on the values of $a_2-a_{10}$, $a_3-a_{11}$, $a_4$, $a_5$,
  $b_2+\frac 1 3(a_8-a_1)$, $b_3+\frac 1 2 a_9$, $b_4-a_6$,
  $b_7-a_{12}$, $b_9$, $b_{10}$, $b_{11}$, $b_{12}$. 

  It follows that the dimension of the space of first-order
  deformations of the point $(C,i)$ is at most $12$. Since, by
  Corollary~\ref{cor:irred_12}, the space $\CM$ is has dimension $12$,
  this concludes the proof.
\end{proof}

\begin{thm}
  The Cohen--Macaulay space $\CM$ of twisted cubics is irreducible,
  smooth and it has dimension $12$.
\end{thm}
\begin{proof}
  We only have to show that $\CM$ is smooth, that is, the tangent
  space has dimension $12$ at every point. We have seen in
  Corollary~\ref{cor:irred_12}, that the open subscheme $U$ of $\CM$,
  consisting of all points $(C,i)$ where $i$ is a closed immersion,
  is smooth.  Hence it remains to consider the most specialized ones
  among the remaining points, namely those having a triple line as the
  scheme-theoretic image. This case was treated in
  Proposition~\ref{prop:triple_smooth}.
\end{proof}
\providecommand{\bysame}{\leavevmode\hbox to3em{\hrulefill}\thinspace}
\providecommand{\MR}{\relax\ifhmode\unskip\space\fi MR }
\providecommand{\MRhref}[2]{%
  \href{http://www.ams.org/mathscinet-getitem?mr=#1}{#2}
}
\providecommand{\href}[2]{#2}

\end{document}